\documentclass[11pt]{amsart}

\usepackage{amssymb}

\newcommand{\A}{\mathcal{A}}

\newcommand{\ot}{\mathrm{ot}}

\newcommand{\ro}{\mathrm{ro}}
\newcommand{\baafa}{BA\!AFA}
\newcommand{\id}{\mathrm{id}}

\newcommand{\forclass}{\mathcal{S}}
\newcommand{\sigmaclosedccc}{$\sigma$-closed\,$*$\,ccc}

\newtheorem{theorem}{Theorem}[section]
\newtheorem{lemma}[theorem]{Lemma}
\newtheorem{sublemma}[theorem]{Sublemma}
\newtheorem{question}[theorem]{Question}
\newtheorem{proposition}[theorem]{Proposition}
\newtheorem{corollary}[theorem]{Corollary}

\theoremstyle{definition}
\newtheorem{definition}[theorem]{Definition}

\title{Bounded forcing axioms and Baumgartner's conjecture}

\author[D. Asper\'o]{David Asper\'o}\thanks{One visit of
  David Asper\'{o} at the KGRC in Vienna was supported by
  the INFTY programme of the ESF through an exchange grant
  and by the University of Vienna through a short research
  contract.}

\address{David Asper\'o, Departamento de Matem\'aticas,
  Universidad Nacional de Colombia, Av. Cra 30 \# 45-03,
  Bogot\'a, Colombia }

\email{daasperohe@unal.edu.co}

\author[S.-D. Friedman]{Sy-David Friedman}\thanks{Sy-David
  Friedman, Miguel Angel Mota and Marcin Sabok would like to
  thank the FWF (the Austrian Science Fund) for its support
  through the grant no P 22430-N13.}

\address{Sy-David Friedman, Kurt G\"odel Research Center for
  Mathematical Logic, W\"ahringer Stra\ss e 25, 1090 Wien,
  Austria}

\email{sdf@logic.univie.ac.at}

\author[M.A. Mota ]{Miguel Angel Mota} 

\address{Miguel Angel Mota, Kurt G\"odel Research Center for
  Mathematical Logic, W\"ahringer Stra\ss e 25, 1090 Wien,
  Austria}

\email{motagaytan@gmail.com}

\author[M. Sabok]{Marcin Sabok}\thanks{Marcin Sabok was
  partially supported by MNiSW (the Polish Ministry of
  Science and Higher Education) through grant number N N201
  418939. Two visits of Marcin Sabok at the Kurt G\"odel
  Research Center in Vienna were supported by the the INFTY
  programme of the European Science Foundation through
  exchange grants no 2535 and 3010 and by the
  Austrian-Polish bilateral grant no 01/2009.}

\address{Marcin Sabok, Instytut Matematyczny Uniwersytetu
  Wroc\l awskiego, pl. Grunwaldzki 2/4, 50-384 Wroc\l aw,
  Poland and Instytut Matematyczny Polskiej Akademii Nauk,
  ul. \'Sniadeckich 8, 00-956 Warszawa, Poland}

\email{sabok@math.uni.wroc.pl}

\begin{document}

\maketitle

\begin{abstract}
  We study the spectrum of forcing notions between the
  iterations of $\sigma$-closed followed by ccc forcings and
  the proper forcings. This includes the hierarchy of
  $\alpha$-proper forcings for indecomposable countable
  ordinals $\alpha$ as well as the Axiom A forcings. We
  focus on the bounded forcing axioms for the hierarchy of
  $\alpha$-proper forcings and connect them to a hierarchy
  of weak club guessing principles. We show that they are,
  in a sense, dual to each other. In particular, these weak
  club guessing principles separate the bounded forcing
  axioms for distinct countable indecomposable ordinals. In
  the study of forcings completely embeddable into an
  iteration of $\sigma$-closed followed by ccc forcing, we
  present an equivalent characterization of this class in
  terms of Baumgartner's Axiom A. This resolves a well-known
  conjecture of Baumgartner from the 1980's.
\end{abstract}

\section{Introduction}
\label{sec:intro}

After the discovery of finite support iteration
\cite{solovay:tennenbaum} and Martin's Axiom
\cite{martin:solovay}, the technique of iterated forcing was
dramatically extended through consideration of iterations
with countable support. The classical paper of Baumgartner
and Laver \cite{baumgartner:laver} on countable support
iterations of Sacks forcing was developed further by
Baumgartner into the theory of Axiom A forcing
\cite{baumgartner}. Baumgartner's Axiom A captures many of
the common features of ccc, $\sigma$-closed and tree-like
forcings and is sufficient to guarantee that $\omega_1$ is
not collapsed in a countable support iteration.  The more
general theory of proper forcing was later developed by
Shelah \cite{shelah:proper} and has replaced Axiom A as the
central notion in the theory of iterated forcing with
countable support.

Together with the introduction of proper forcing, Shelah
also considered the notion of $\alpha$-proper forcing
\cite[Chapter V]{shelah:proper} for indecomposable countable
ordinals $\alpha$. Forcings which are $\alpha$-proper for
all countable ordinals are called $<\!\!\omega_1$-proper.
Years later, Ishiu \cite{ishiu} proved the striking result
that the notions of Axiom A and $<\!\!\omega_1$-properness
are in fact the same, meaning that, up to
forcing-equivalence, they describe the same classes of
quasi-orders.  This also explained an earlier result of
Koszmider \cite{koszmider} saying that Axiom A is preserved
by countable-support iteration.

Baumgartner showed that the analogue of Martin's Axiom for
proper forcing, called PFA (the Proper Forcing Axiom) is
consistent relative to a supercompact cardinal and it is
conjectured that its consistency strength is exactly that.
PFA and the forcing axioms for the classes of
$\alpha$-proper forcings (written as PFA$_\alpha$) were
later systematically studied by Shelah \cite{shelah:proper}.
However, a still very useful weakening of PFA considered by
Goldstern and Shelah \cite{goldstern:shelah} and called BPFA
(the Bounded Proper Forcing Axiom) turned out to have much
lower consistency strength, below that of a Mahlo cardinal.
In addition, some important consequences of PFA, such as the
Todor\v{c}evi\'c--Veli\v{c}kovi\'c result that c =
$\aleph_2$ holds under PFA, were shown to also follow from
BPFA \cite{moore}. On the other hand, one should remember
that the proof of Todor\v{c}evi\'c--Veli\v{c}kovi\'c in fact
only uses FA(\sigmaclosedccc), i.e. the forcing axiom for
the class of forcings completely embeddable into an
iteration of $\sigma$-closed followed by ccc forcing. We
will say that a forcing is \textit{embeddable into
  \sigmaclosedccc} if it is forcing-equivalent to a forcing
which can be completely embedded into an iteration of
$\sigma$-closed followed by ccc forcing.

Given a class of forcing notions $\forclass$, the
\textit{Bounded Forcing Axiom for $\forclass$}, denoted by
BFA($\forclass$), is the statement that for each complete
Boolean algebra $B$ in $\forclass$ and any collection
$\mathcal{D}$ of $\omega_1$-many size at most $\omega_1$
predense subsets of $B$, there is a filter on B which
intersects each element of $\mathcal{D}$. An equivalent form
of of BFA($\forclass$), due independently to Bagaria
\cite{bagaria} and Stavi--V\"a\"an\"anen
\cite{stavi:vaananen}, states that $H(\omega_2)^V$ is
$\Sigma_1$-elementary in $H(\omega_2)^{V^B}$ for any
complete Boolean algebra $B$ in $\forclass$.

The Bounded Forcing Axiom for the class of ccc forcing
notions is equivalent to Martin's Axiom and the Bounded
Forcing Axiom for the class of proper forcings is exactly
BPFA. In fact, there is a whole spectrum of forcing axioms,
namely the Bounded Forcing Axioms for the classes of
$\alpha$-proper forcing notions (written as BPFA$_\alpha$),
where $\alpha$ can be any countable indecomposable ordinal.
There is also the Bounded Forcing Axiom for
$<\!\!\omega_1$-proper forcing notions, which is (a priori)
weaker than all BPFA$_\alpha$. By the result of Ishiu, it is
equivalent to the Bounded Forcing Axiom for the class of
Axiom A forcings, also denoted by \baafa. A still (a priori)
weaker variation is the Bounded Forcing Axiom for the class
of forcings embeddable into \sigmaclosedccc. We denote this
axiom by BFA(\sigmaclosedccc). Remarkably,
Todor\v{c}evi\'c showed (see \cite{todorcevic} or
\cite[Lemma 2.4]{aspero}) that the consistency strength of
BFA(\sigmaclosedccc) is the same as of BPFA, i.e. a
reflecting cardinal. This implies that actually all the
axioms along this hierarchy have the same consistency
strength.

In \cite{weinert} Weinert showed that \baafa\ is strictly
weaker than BPFA, relative to a reflecting cardinal. In this
paper we separate the axioms BPFA$_\alpha$ for all
indecomposable countable ordinals. We consider a hierarchy
of weak club guessing principles TWCG$_\alpha$ (for a
definition see Section \ref{sec:separation}) and show the
following
\begin{theorem}\label{separation}
  For indecomposable ordinals $\alpha<\beta<\omega_1$ the
  axiom \textnormal{BPFA}$_\alpha$ (or
  \textnormal{PFA}$_\alpha$) is consistent with
  \textnormal{TWCG}$_\alpha$, relative to a reflecting
  cardinal (or a supercompact), whereas
  \textnormal{BPFA}$_\alpha$ is inconsistent with
  \textnormal{TWCG}$_\beta$.
\end{theorem} 
The weak club guessing principles were introduced already by
Shelah, who considered them as a variant of the full (or
tail) club guessing principles (cf. \cite{ishiu}). Theorem
\ref{separation} actually refines the separation of the
axioms PFA$_\alpha$ due to Shelah \cite[Chapter
XVII]{shelah:proper}, which was done in terms of the full
club guessing principles. We also show the following.
\begin{theorem}\label{separationprinciples}
  For indecomposable ordinals $\alpha<\beta<\omega_1$, the
  principle \textnormal{TWCG}$_\beta$ implies
  \textnormal{TWCG}$_\alpha$ and \textnormal{TWCG}$_\alpha$
  does not imply \textnormal{TWCG}$_\beta$.
\end{theorem}

The role of the forcings embeddable into \sigmaclosedccc\
was already recognized by Baumgartner, who actually
conjectured that every forcing satisfying Axiom A can be
embedded into an iteration of a $\sigma$-closed followed by
a ccc forcing. This would of course mean that the two
classes are in fact the same, up to forcing-equivalence.
Probably, the first motivation came with the Mathias forcing
and its decomposition into $P(\omega)\slash\mathrm{fin}$
followed by the Mathias forcing with an ultrafilter. Later,
the conjecture was confirmed for the Sacks forcing and other
tree-like forcing notions in \cite{groszek:jech}. Miyamoto
\cite{miyamoto} proved it for the iterations of a ccc
followed by a $\sigma$-closed forcing. Recently, Zapletal
proved that in most cases if an idealized forcing is proper,
then it is in fact embeddable into \sigmaclosedccc\
\cite[Theorems 4.1.5, 4.2.4, 4.3.26, 4.5.9, Lemma
4.7.7]{zapletal:book}.

We introduce the notion of a strong Axiom A forcing (for a
precise definition see Section \ref{sec:embeddable}), which
is basically saying that a forcing satisfies Axiom A after
taking a product with any $\sigma$-closed forcing. We prove
the following characterization.

\begin{theorem}\label{strongaa}
  Let $P$ be a forcing notion. The following are equivalent
  \begin{itemize}
  \item[(i)] $P$ satisfies strong Axiom A,
  \item[(ii)] $P$ is embeddable into \sigmaclosedccc.
  \end{itemize}
\end{theorem}

Theorem \ref{strongaa} is in fact a confirmation of
Baumgartner's conjecture as it says that there indeed is a
close connection between Axiom A and embeddability into
$\sigma$-closed and ccc. This characterization cannot,
however, be strengthened to the one suggested by Baumgartner
because Theorem \ref{strongaa} leads also to the following
counterexample.

\begin{corollary}\label{original}
  There is an Axiom A forcing notion which is not embeddable
  into \sigmaclosedccc. It is of the form
  ccc\,$*$\,$\sigma$-closed\,$*$\,ccc.
\end{corollary}

This paper is organized as follows. Section
\ref{sec:separation} contains the results on the weak club
guessing principles and the bounded forcing axioms for
$\alpha$-proper forcings. Section \ref{sec:embeddable}
contains the characterization of forcings embeddable into
\sigmaclosedccc.

\subsection{Remark}

\label{sec:ack}
After this work has been done, we have learnt that
Todor\v{c}evi\'c can also derive Corollary \ref{original}
from the results of his \cite[Section
2]{todorcevic:relations}; this proof has, however, never
been published.

\section{Bounded forcing axioms and weak club guessing}
\label{sec:separation}

\begin{definition}
  Let $\kappa > \omega$ be a regular cardinal, $\alpha$ an
  ordinal and $\mathcal{M}=\{M_{\varepsilon} :\varepsilon
  \in \alpha \}$ be a sequence of countable elementary
  substructures of $H(\kappa) $. We say that $\mathcal{M}$
  is an \emph{internally approachable tower} if the
  following hold:
  \begin{itemize}
  \item[(i)] $\{M_{\varepsilon} :\varepsilon \leq \eta \}
    \in M_{\eta+1}$ for every $\eta \in \alpha$ with $\eta+1
    \in \alpha$,
  \item[(ii)] $M_{\eta} = \bigcup \{M_{\varepsilon}
    :\varepsilon < \eta \}$ for every limit ordinal $\eta
    \in \alpha$.
  \end{itemize}
\end{definition}

As usual, $H(\kappa)$ is the collection of all sets of
hereditary cardinality less than $\kappa$. We will identify
$H(\kappa)$ with the structure $\langle H(\kappa), \in, \lhd
\rangle$, where $\lhd$ is a fixed well order of $H(\kappa)$.

\begin{definition}
  Let $P$ be a partial order and $\alpha$ a countable
  ordinal.
\begin{itemize}
\item[(a)] Given $q \in P$ and
  $\mathcal{M}=\{M_{\varepsilon} :\varepsilon \in \alpha \}$
  an internally approachable tower of countable elementary
  substructures of $H(\kappa)$ with $P \in M_{0}$, we say
  that $q$ is \textit{generic over $\mathcal{M}$} if $q$
  forces that $\dot{G} \cap M_{\varepsilon}$ is generic over
  $M_{\varepsilon}$ for every $\varepsilon \in \alpha$.
\item[(b)] $P$ is $\alpha$-\emph{proper} if for every
  sufficiently large regular cardinal $\kappa$, for every
  internally approachable tower
  $\mathcal{M}=\{M_{\varepsilon} :\varepsilon \in \alpha \}$
  as above and for every condition $p \in P \cap M_{0}$,
  there exists $q \leq p$ such that $q$ is $(\mathcal{M},
  P)$-generic. $P$ is $<\!\!\omega_1$-proper if it is
  $\alpha$-proper for each $\alpha<\omega_1$.
\end{itemize}
\end{definition}

Note that if $P$ is proper (i.e., 1-proper), then $P$ is
$n$-proper for every natural number $n$. Recall that a
countable ordinal $\beta$ is said to be
\emph{indecomposable} if there exists a nonzero ordinal
$\tau$ such that $\beta = \omega^{\tau}$ (this is ordinal
exponentiation). Equivalently, $\beta$ is indecomposable if
for every $\gamma < \beta$, the order type of the interval
$(\gamma, \beta)$ is equal to $\beta$.  Now, if $P$ is
$\alpha$-proper and $\beta$ is the first indecomposable
ordinal above $\alpha$, then $P$ is $\gamma$-proper for every
$\gamma < \beta$.

Let $\alpha$ be an indecomposable ordinal. We denote by
PFA$_{\alpha}$ the forcing axiom for the class of
$\alpha$-proper forcing notions.  By BPFA$_{\alpha}$ we
denote the bounded forcing axiom for this class.

\begin{definition}
  An $\alpha$-\emph{ladder system} is a sequence $\bar{A}=
  \langle A_\beta: \beta<\omega_1\rangle$ such that for each
  $\beta<\omega_1$, with $\alpha$ dividing $\beta$, the set
  $A_\beta$ is a closed unbounded subset of $\beta$ and
  $\ot(A_\beta) = \alpha$. We will always assume that
  $\langle A_\beta(\tau): \tau< \alpha\rangle$ is the
  increasing enumeration of the elements of $A_\beta$. We
  say that an $\alpha$-ladder system $\langle A_\beta:
  \beta<\omega_1\rangle$ is \textit{thin} if for any
  $\beta<\omega_1$ the set
  $\{A_\gamma\cap\beta:\gamma\in\omega_1\}$ is countable.
\end{definition}
\begin{definition}    
  The \textit{$\alpha$-Weak Club Guessing} principle,
  denoted by WCG$_\alpha$ says that there is an
  $\alpha$-ladder system $\bar{A}$ such that for every club
  $D\subseteq\omega_1$ there is $\beta\in D$ such that
  $\alpha$ divides $\beta$ and $\ot(A_\beta \cap D) =
  \alpha$. The \textit{$\alpha$-Thin Weak Club Guessing}
  principle, denoted by TWCG$_\alpha$, also asserts the
  existence of such an $\bar{A}$ but with the additional
  requirement of being thin.
\end{definition}

Thin (full) club guessing ladder systems have been
considered in the literature in \cite{zapletal,ishiu}.
Zapletal mentions \cite[Section 1.A]{zapletal} that their
existence can be derived from $\Diamond$ and shows
\cite[Section 2]{zapletal} how to force one with a
$\sigma$-closed forcing notion.

\begin{theorem}\label{wcg}
  For indecomposable ordinals $\alpha<\beta<\omega_1$,
  \textnormal{BPFA}$_\alpha$ implies the negation of
  \textnormal{TWCG}$_\beta$.
\end{theorem}
\begin{proof}
  By the $\Sigma_1(H(\omega_2))$ generic absoluteness
  characterization of BPFA$_\alpha$, it suffices to prove
  that for any thin $\beta$-club guessing sequence $\bar A$
  there is an $\alpha$-proper forcing notion shooting a club
  in $\omega_1$ which is not guessed by $\bar A$.

  Fix a thin $\beta$-club guessing sequence $\bar A=\langle
  A_\gamma:\gamma<\omega_1\rangle$. Let $P$ be the following
  forcing notion. Conditions in $P$ are countable subsets
  $C$ of $\omega_1$ such that
  \begin{itemize}
  \item $C$ is closed in the order topology,
  \item $\ot(C\cap A_\gamma) < \beta$ for each
    $\gamma<\omega_1$ with $\beta$ dividing $\gamma$.
  \end{itemize}

  The ordering $\leq_P$ on $P$ is the end-extension. We need
  to show that $P$ is $\alpha$-proper. Let $\kappa$ be a
  sufficiently large regular cardinal and let $\lhd$ be a
  well-ordering on $H(\kappa)$.  Pick an internally
  approachable tower $\mathcal{M}=\langle
  M_\gamma:\gamma<\alpha\rangle$ of countable elementary
  submodels of $\langle H(\kappa),\in,\lhd \rangle$ such
  that $\bar A \in M_0$.  Put
  $\rho_\gamma=M_\gamma\cap\omega_1$.  Let $p\in M_0$ be any
  condition in $P$. We need to find a condition extending
  $p$ and generic for the whole tower. For so doing,
  consider the $\lhd$-least $\omega$-ladder system $\bar B$
  and note that $\bar B \in M_0$. 

  Say that $X \subseteq \omega_1$ is
  \textit{$\mathcal{M}$-accessible} if the order type of $X$
  is strictly less than $\rho_0$ and $X \cap \rho_\gamma \in
  M_{\gamma+1}$ for every $\gamma<\alpha$. Note that each
  $A_\gamma$ is $\mathcal{M}$-accessible by thinness. For
  each $\mathcal{M}$-accessible $X\subseteq\omega_1$ we
  construct by induction a decreasing sequence of conditions
  $p(\gamma, X)$ for $\gamma\leq\alpha$ such that for each
  $\gamma \leq \alpha$ we have
  \begin{itemize}
  \item[(i)] $p(0, X)=p$,
  \item[(ii)] $p(\gamma, X)$ is a $P$-generic condition for
    $\langle M_\delta:\delta<\gamma\rangle$,
  \item[(iii)] if $\gamma=\delta+1$, then $p(\gamma, X) \cap
    (A_{\rho_\delta} \cup X) \subseteq p \cup
    \{\rho_\varepsilon : \varepsilon \leq \delta \}.$
  \item[(iv)] $p(\gamma, X)\in M_\gamma$ for successor
    $\gamma$ and $p(\gamma, X) \in M_{\gamma+1}$ for limit
    $\gamma$ 
  \end{itemize}
  Here $M_{\alpha+1}=H(\kappa)$.  In order to guarantee that
  (iv) holds, we will also require the following conditions
  \begin{itemize}    
  \item[(v)] $p(\gamma,X)=\bigcup_{n
      <\omega}p(B_{\gamma}(n), X\cup A_{\rho_\gamma})$ for
    limit $\gamma$,
  \item[(vi)] $p(\gamma+1, X)=p(\gamma, X)$ for limit $\gamma$,
  \item[(vii)] if $\gamma < \alpha$ is zero or successor,
    then $p(\gamma+1, X)$ is the $\lhd$-least condition
    which extends $p(\gamma,X)$ and satisfies (i), (ii) and
    (iii).
  \end{itemize}

  Put $p(0,X)=p$. Suppose $\gamma\leq\alpha$ and
  $p(\delta,X)$ have been constructed for all
  $\delta<\gamma$. If $\gamma$ is limit, then $p(\gamma, X)$
  is defined as in (v). If $\gamma=\delta+1$ and $\delta$ is
  a limit, then $p(\gamma, X)=p(\delta, X)$. Suppose
  $\gamma=\delta+1$ and $\delta$ is zero or a successor, in
  which case $p(\delta, X)\in M_\delta$. We need to show
  that there exists a condition extending $p(\delta, X)$ and
  satisfying (ii) and (iii).

  Enumerate all dense open subsets of $P$ in $M_\delta$ into
  a sequence $\langle D_n:n<\omega\rangle$ (assume $D_0=P$)
  and inductively construct a decreasing sequence of
  conditions $p^n\in M_\delta\cap D_n$ such that
  $p^0=p(\delta, X)$ and $p^n\cap (A_{\rho_\delta}\cup X)
  =p(\delta,X)$. Suppose $p^n\in M_\delta$ has been
  constructed and let $\eta^n=\sup(p_n)$. Consider the
  function $f:\omega_1 \setminus\eta^n\rightarrow \omega_1$
  defined as follows: for $\nu\in\omega_1 \setminus\eta^n$
  let $q^\nu$ be the $\lhd$-smallest condition which extends
  $p^n\cup\{\nu\}$ and belongs to $D_{n+1}$. Then we define
  $f(\nu)$ as the maximum of $q^\nu$.  Now let
  $E\subseteq\omega_1$ be the club of those points greater
  than $\eta^n$ which are closed under $f$. Note that $f$
  and $E$ are in $M_\delta$, since they are definable from
  parameters in this model.  It follows that
  $$\ot(E \cap \rho_\delta)=\rho_\delta > \ot(A_{\rho_\delta}
  \cup (X \cap \rho_\delta)) $$

  Choose two elements $\nu_0<\nu_1< \rho_\delta$ of $E$ such
  that $[\nu_0,\nu_1]\cap(A_{\rho_\delta}\cup (X \cap
  \rho_\delta)) = \emptyset$. We can choose $p^{n+1}$ to be
  $q^{\nu_0}$.

  Now the condition
  $\bigcup_{n<\omega}p^n\cup\{\rho_\delta\}$ is $P$-generic
  for $M_\delta$ and for the whole subtower $\langle
  M_\varepsilon:\varepsilon<\gamma\rangle$ and satisfies
  (ii) and (iii). Let $p(\gamma, X)$ be the $\lhd$-smallest
  condition with these properties and note that $p(\gamma,
  X) \in M_\gamma$, since this condition is definable (using
  the order $\lhd$ of $H(\kappa)$) from $p$, $X \cap
  \rho_\delta$ and $\langle
  M_\varepsilon:\varepsilon<\gamma\rangle$. This ends the
  successor step of the inductive construction. It is
  immediate that the condition $p(\alpha, \emptyset)$ is
  generic for the whole tower $\langle
  M_\gamma:\gamma<\alpha\rangle$.
\end{proof}

The following proposition (due to Shelah) appears in
\cite[Proposition 3.5]{ishiu} for full club guessing ladder
systems. The proof for weak club guessing is exactly the
same. We provide it for the reader's convenience.

\begin{proposition}
  Let $\bar A=\langle A_\gamma:\gamma<\omega_1\rangle$ be a
  thin $\beta$-ladder system and $P$ a $\beta$-proper notion
  of forcing. If $\bar A$ witness \textnormal{TWCG}$_\beta$,
  then $\bar A$ witnesses \textnormal{TWCG}$_\beta$ in any
  generic extension with $P$.
\end{proposition}
\begin{proof}
  Let $\dot{E}$ be a $P$-name for a club and $p$ a condition
  in $P$. It suffices to prove that there exists an ordinal
  $\rho^\ast$ and condition $q \leq p$ such that $q$ forces
  that the intersection of $\dot{E}$ with $A_{\rho^{\ast}}$
  has order type equal to $\beta$. For so doing, let
  $\kappa$ be a sufficiently large regular cardinal and
  consider an internally approachable tower
  $\mathcal{M}=\langle M_{\varepsilon} : \varepsilon \in
  \omega_{1} \rangle $ of countable elementary substructures
  of $H(\kappa)$ such that $\bar A$, $P$, $\dot{E}$ and $p$
  are in $M_{0}$. Let $F$ be the club of those countable
  ordinals $\rho$ such that $\rho= M_{\rho} \cap \omega_1$.
  Now, by TWCG$_\beta$ (applied in $V$), there exist
  $\rho^{\ast} \in F$ such that $\ot(A_{\rho^{\ast}} \cap
  F)=\beta$. Note that for each $\rho \in A_{\rho^{\ast}}
  \cap F$, any $(M_{\rho}, P)$-generic condition forces that
  $\rho \in \dot{E}$. So, it suffices to prove that there is
  a condition extending $p$ which is generic for all
  elements of the tower $\mathcal{M}^\ast = \langle
  M_{\varepsilon} : \varepsilon \in A_{\rho^{\ast}} \cap F
  \rangle $. Given that $P$ is $\beta$-proper, this can be
  reduced to proving that $\mathcal{M}^\ast$ is internally
  approachable, which is true by the assumptions that $\bar
  A$ is thin and $\mathcal{M}$ is internally approachable.
\end{proof}

\begin{corollary}\label{consistency}
  For every indecomposable ordinal $\gamma<\omega_1$ the
  principle \textnormal{TWCG}$_\gamma$ is consistent with
  \textnormal{BPFA}$_\gamma$ (or \textnormal{PFA}$_\gamma$),
  relative to a reflecting cardinal (or a supercompact).
\end{corollary}
\begin{proof}
  We prove only the PFA version. The proof is very similar
  to the usual proof of the consistency of PFA, and so we
  omit the details. We start with a ground model with a
  supercompact satisfying TWCG$_\gamma$ (there is one by the
  results of \cite{zapletal}). The generic extension that we
  need is obtained by a countable support iteration of
  length a supercompact cardinal, where in each step of the
  iteration we only consider names for $\gamma$-proper
  partial orders.  Since the countable support iteration of
  $\gamma$-proper forcing notions is $\gamma$-proper
  \cite[Chapter 5, Theorem 3.5]{shelah:proper} and
  $\gamma$-proper forcing preserves TWCG$_\gamma$, we get a
  model of both, PFA$_\gamma$ and TWCG$_\gamma$.
\end{proof}

Together, Theorem \ref{wcg} and Corollary \ref{consistency}
prove Theorem \ref{separation}. The separation of the axioms
PFA$_\alpha$ for indecomposable ordinals $\alpha<\omega_1$
appears in Shelah's \cite[Chapter XVII, Remark
3.15]{shelah:proper}.  We are not aware, however, if the
separation with the bounded versions has ever appeared in
the literature, so we mention it in the following corollary.

\begin{corollary}
  For indecomposable ordinals $\alpha<\beta<\omega_1$,
  \textnormal{BPFA}$_\beta$ (or \textnormal{PFA}$_\beta$)
  does not imply \textnormal{BPFA}$_\alpha$, relative to a
  reflecting cardinal (or a supercompact).
\end{corollary}
\begin{proof}
  By Corollary \ref{consistency} there is a model of
  BPFA$_\beta$ (or PFA$_\beta$) and TWCG$_\beta$, relative
  to a reflecting cardinal (or a supercompact). It cannot
  satisfy BPFA$_\alpha$ by Theorem \ref{wcg}.
\end{proof}

In the remaining part of this section we will prove Theorem
\ref{separationprinciples}. We will need an additional piece
of notation. Given an indecomposable ordinal $\beta$ and a
cardinal $\kappa \leq \omega_1$, a $(\beta,
\kappa)$-\emph{system} is a sequence $\bar{A}= \langle
A^{\alpha}_{\delta} : \alpha \in \kappa, \delta \in \omega_1
\rangle$ such that for every $\alpha$ and $\delta$, with
$\beta$ dividing $\delta$, the set $A^{\alpha}_{\delta}$ is
a closed unbounded subset of $\delta$ of order type $\beta$.
A $(\beta,\kappa)$-system $\bar A$ is \textit{thin} if for
any $\gamma\in\omega_1$, the set
$\{A^\alpha_\delta\cap\gamma : \alpha<\kappa,
\delta\in\omega_1\}$ is countable.

Note that a $(\beta,\kappa)$-system $\bar A$ can be
enumerated as $(\bar{A}_\delta : \delta<\omega_1)$, but then
we must remember that $\bar{A}_\delta$ need not be cofinal
in $\delta$. Such enumerations will be used in the proof of
Theorem \ref{separationprinciples} below.

The principle WCG$^{\kappa}_{\beta}$ asserts the existence
of a $(\beta, \kappa)$-system $\bar{A}= \langle
A^{\alpha}_{\delta} : \alpha \in \kappa, \delta \in \omega_1
\rangle$ such that for every club $D \subseteq \omega_{1}$,
there exists $\delta \in D$ and $\alpha \in \kappa$ such
that $\beta$ divides $\delta$ and $ot(A^{\alpha}_{\delta}
\cap D)= \beta$. The principle TWCG$^{\kappa}_{\beta}$ says
exactly the same that WCG$^{\kappa}_{\beta}$ with the
additional requirement that $\bar{A}$ must be thin.

\begin{lemma}\label{omegalemma}
  For any indecomposable ordinal $\beta$,
  \textnormal{TWCG}$_\beta$ is equivalent to
  \textnormal{TWCG}$^{\aleph_0}_{\beta}$ and the same holds
  for the non-thin versions.
\end{lemma}

\begin{proof}
  The two statements have the same proof. We only focus on
  the thin versions and we show that
  TWCG$^{\aleph_0}_{\beta}$ implies TWCG$_\beta$.  So, let
  $\langle A^{n}_{\delta} : n \in \omega, \delta \in
  \omega_1 \rangle$ be a $(\beta, \aleph_0)$-system
  witnessing TWCG$^{\aleph_0}_{\beta}$. We define a thin
  $\beta$-ladder system $\langle B_{\delta} : \delta \in
  \omega_1 \rangle$ as follows. First, for each $\delta$
  divisible by $\beta$ fix a cofinal sequence $\langle
  \delta_{n} : n \in \omega \rangle \subseteq \delta$ of
  order type $\omega$. Define $B_{\delta}=
  \bigcup\{B^{n}_{\delta} : n \in \omega \}$, where
  $B^{n}_{\delta}$ is equal to $A^{n}_{\delta}\setminus
  \delta_{n}$ . Now $\langle B_{\delta} : \delta \in
  \omega_1 \rangle$ is a thin system. To see this, notice
  that for each $\gamma\in\omega_1$ if $\delta>\gamma$,
  $\delta\in\omega_1$ is divisible by $\beta$, then only
  finitely many of $\delta_n$'s are below $\gamma$ and hence
  $B_\delta\cap\gamma$ is a union of finitely many of the
  sets $A^n_\delta\cap\gamma\setminus\delta_n$. The fact
  that $\langle B_{\delta} : \delta \in \omega_1 \rangle$
  witnesses TWCG$_\beta$ follows directly from the
  assumption that $\langle A^{n}_{\delta} : n \in \omega,
  \delta \in \omega_1 \rangle$ witnesses
  TWCG$^{\aleph_0}_{\beta}$
\end{proof}

Now we are ready to prove Theorem
\ref{separationprinciples}.

\begin{proof}[Proof of Theorem \ref{separationprinciples}]
  The fact that TWCG$_\alpha$ does not imply TWCG$_\beta$
  follows directly from Theorem \ref{separation}.
  Alternately, to derive this just in ZFC, one can start
  with a model of TWCG$_\alpha$ + CH +
  $2^{\aleph_1}=\aleph_2$ and then force with a
  countable-support iteration of length $\omega_2$ of
  $\alpha$-proper forcings, killing all $\beta$-thin club
  sequences.

  Now we prove that TWCG$_\beta$ implies TWCG$_\alpha$.
  Assume TWCG$_\alpha$ fails. We will show by induction on
  $\alpha'\in[\alpha,\beta]$ that TWCG$_{\alpha'}$ fails.
  In fact, we will show that if
  $\bar{A}=\{\bar{A}_{\delta}:\delta\in\omega_1\}$
  enumerates a thin $(\alpha', \aleph_0)$-system, then there
  exists a club $D$ such that for every $\delta\in\omega_1$
  the intersection of $\bar{A}_{\delta}$ with $D$ has order
  type strictly less than $\alpha$. The case
  $\alpha'=\alpha$ follows from Lemma \ref{omegalemma}.
  Assume $\alpha'>\alpha$ and fix an enumeration
  $\{\bar{A}_{\delta}:\delta\in\omega_1\}$ of a thin
  $(\alpha', \aleph_0)$-system $\bar{A}$. For each ordinal
  $\delta$ find an increasing cofinal sequence $\{\delta_n:
  n\in\omega\}\subseteq\bar{A}_{\delta}$ of limit points of
  the set $\bar{A}_{\delta}$ such that the order types of
  $$A'(\delta, 0)=\bar{A}_{\delta}\cap\delta_{0}$$ and
  $$A'(\delta, n+1)=(\bar{A}_\delta\cap\delta_{n+1})
  \setminus\delta_n$$ are indecomposable ordinals greater
  than or equal to $\alpha$.  Now, consider the thin system
  enumeration $$A'=\{A'(\delta,n):\delta\in\omega_1,
  n\in\omega\},$$ and note that for each indecomposable
  $\pi$ in the semi-open interval $[\alpha,\alpha')$, the
  inductive hypothesis ensures the existence of a club
  $C_{\pi}$ such that for every $\delta$ and for every $n$
  if the order type of $A'(\delta,n)$ is equal to $\pi$,
  then $\ot(A'(\delta, n) \cap C_{\pi})<\alpha$. Note that
  if $\bar{A}$ is thin, then the set of those elements of
  $A'$ whose order type is equal to $\pi$ is a $(\pi,
  \aleph_0)$-system. Let $C$ be the intersection of all the
  $C_{\pi}$.  Now define the set $\bar{B}_{\delta}$ as
  follows
  $$\bar{B}_{\delta}=\{ \delta_{n} : n\in\omega\} \cup
  \bigcup \{A'(\delta, n) \cap C : n \in \omega\}.$$

  Note that this set has order type at most $\alpha$. Also
  note that if $\gamma<\sup\bar{A}_\delta$, then
  $\bar{B}_\delta\cap\gamma$ is equal to the union of
  $\bar{A}_\delta\cap\gamma\cap C$ together with a finite
  subset of $\sup\bar{A}_\delta$. Therefore the system
  $\bar{B}=\{\bar{B}_\delta:\delta\in\omega_1\}$ is thin.
  Finally, find a club D subseteq C witnessing that the
  system $\bar{B}$ does not guess in the $(\alpha,
  \aleph_0)$-sense. Now $D$ is as desired.
\end{proof}

\section{Forcings embeddable into \sigmaclosedccc}
\label{sec:embeddable}

Recall that a forcing notion $P$ satisfies the
\textit{uniform Axiom A} if there is an ordering $\leq_0$ on
$P$ refining its original ordering such that any
$\leq_0$-descending $\omega$-sequence has a $\leq_0$-lower
bound and for any antichain $A$ in $P$ any condition can be
$\leq_0$-extended to become compatible with at most
countably many elements of $A$. By a \textit{quasi-order} we
mean a reflexive and transitive relation.

Ishiu showed \cite[Theorem 4.3]{ishiu} that, up to
forcing-equivalence, Axiom A and uniform Axiom A are
equivalent and describe precisely the class of
$<\!\!\omega_1$-proper quasi-orders.  More precisely, he
showed that if $P$ is an Axiom A forcing notion, then there
is a quasi-order $P'$ which is forcing-equivalent to $P$ and
an ordering $\leq_0$ on $P'$ such that $P'$ satisfies the
uniform Axiom A via $\leq_0$.  This is a motivation for the
following definition.

\begin{definition}
  A forcing notion $P$ satisfies \textit{strong Axiom A} if
  there a quasi-order $P'$, forcing-equivalent to $P$, with
  an ordering $\leq_0$ on $P'$ such that for any
  $\sigma$-closed forcing $S$ the product $S\times P'$
  satisfies uniform Axiom A via $\leq_S\times\leq_0$.
\end{definition}

Any forcing of the form $R*\dot Q$, where $R$ is
$\sigma$-closed and $\dot Q$ is forced to be ccc, satisfies
the uniform Axiom A.  The ordering $\leq_0$ on $R*\dot Q$ is
simply $\leq_R\!\times\dot =$, i.e. $(r_1,\dot
q_1)\leq_0(r_0,\dot q_0)$ if $r_1\leq_R r_0$ and $r_1\Vdash
\dot q_0=\dot q_1$. To see that $\leq_0$ witnesses the
uniform Axiom A, take an antichain $\A$ in $R*\dot Q$ and a
condition $(r_0,\dot q_0)\in R*\dot Q$. Pick any $R$-generic
filter $G$ over $V$ through $r_0$ and note that in $V[G]$ we
have that $\{(\dot q)\slash G:\exists r\in G\ (r,\dot
q\slash G)\in\A\}$ is an antichain in $\dot Q\slash G$ and
hence it is countable by the assumption that $R\Vdash\dot Q$
is ccc.  Note that for each $(\dot q)\slash G$ in the above
set, there is only one $r\in G$ such that $(r,\dot q)\in\A$,
as $\A$ is an antichain.  Since $R$ does not add new
countable subsets of the ground model, there is a countable
$\A_0\subseteq\A$ in $V$ such that for some condition $r'\in
G$ we have
\begin{displaymath}\label{forcing}
  r'\Vdash\{(r,\dot q)\in\A: r\in\dot G\}=\A_0.\tag{$*$}
\end{displaymath}
Enumerate $\A_0$ as $\{(r_n',\dot q_n):n<\omega\}$. Since
$r_0,r'$ and all the $r'_n$ are in $G$, we can find $r_1\in
R$ extending all these conditions. Now we have that
$(r_1,\dot q_0)\leq_0 (r_0,\dot q_0)$ and it is enough to
check that $\{(r,\dot q)\in\A: (r,\dot q)$ is compatible
with $(r_1,\dot q_0)\}$ is contained in $\A_0$. But if
$(r'',\dot q'')\in\A\setminus\A_0$ were compatible with
$(r_1,\dot q_0)$, then forcing with a filter $G$ such that
$r'',r_1\in G$ would give that $(r'',\dot q'')\in\{(r,\dot
q)\in\A:r\in G\}$, contradicting (\ref{forcing}).

Recall that if $A$ is a complete Boolean algebra and $B$ is
a complete Boolean subalgebra of $A$, then the
\textit{projection} $\pi:A\rightarrow B$ is defined as
follows: $\pi(a)=\bigwedge\{b\in B: a\leq b\}$, where the
Boolean operation is computed in either of the two Boolean
algebras.

Now we prove Theorem \ref{strongaa}

\begin{proof}[Proof of Theorem \ref{strongaa}]
  (ii)$\Rightarrow$(i). Suppose $P\lessdot R*\dot Q$, where
  $R$ is $\sigma$-closed and $\dot Q$ is forced to be ccc.
  Without loss of generality assume that $P$ is a complete
  Boolean subalgebra of $\ro(R*\dot Q)$ and let
  $\pi:\ro(R*\dot Q)\rightarrow P$ be the projection. Let
  $$P'=\{(p,(r,\dot q)):p\in P, (r,\dot q)\in R*\dot Q\mbox{
    and } p\wedge (r,\dot q)\not=0\},$$ where the Boolean
  operation is computed in $\ro(R*\dot Q)$. Consider the
  function $\pi':P'\rightarrow P$ defined as:
  $$\pi'((p,(r,\dot q)))=p\wedge\pi((r,\dot q))$$ and define
  the order $\leq_{P'}$ on $P'$ as follows: $(p_1,(r_1,\dot
  q_1))\leq_{P'}(p_0,(r_0,\dot q_0))$ if
  $\pi'((p_1,(r_1,\dot q_1)))\leq_{P}\pi'((p_0,(r_0,\dot
  q_0)))$. Thus $P'$ becomes a quasi-order with $\leq_{P'}$.
  Note that the definition of $\leq_{P'}$ implies that the
  function $\pi'$ is a dense embedding from $P'$ to $P$,
  hence $P'$ and $P$ are forcing-equivalent.

  Recall that on $R*\dot Q$ we have the natural ordering
  $\leq_R\!\times\dot=$ (see remarks preceeding this
  theorem) to witness uniform Axiom A. Let $\leq_0$ on $P'$
  be defined as follows: $(p_1,(r_1,\dot
  q_1))\leq_0(p_0,(r_0,\dot q_0))$ if $p_1=p_0$, $r_1\leq_R
  r_0$ and $r_1\Vdash \dot q_1=\dot q_0$. Now we claim that
  this $\leq_0$ witnesses the strong Axiom A.

  Let $S$ be a $\sigma$-closed forcing notion. We need to
  check that $S\times P'$ satisfies uniform Axiom A via
  $\leq_S\times\leq_0$. It is clear that $S\times P'$ is
  $\sigma$-closed with respect to $\leq_S\times\leq_0$. Take
  an antichain $A$ in $S\times P'$, $s\in S$ and $(p,(r,\dot
  q))\in P'$. Via $\id\times \pi'$ we get an antichain $A'$
  in $S\times P$. As every element of $\ro(R*\dot Q)$ is a
  supremum of an antichain in $R*\dot Q$, we can refine the
  antichain $A'$ to an antichain $A''$ such that 
  \begin{itemize}
  \item[(a)] every element of $A''$ is of the form
    $(s,(r,\dot q))$ for some $s\in S$ and $(r,\dot q)\in
    R*\dot Q$,
  \item[(b)] every element of $A'$ is the supremum of a
    subset of $A''$.
  \end{itemize} 
  Now, $A''$ is an antichain in $S\times(R*\dot Q)$.  The
  latter is the same as $(S\times R)*\dot Q$ (where $\dot
  Q$, as an $R$-name naturally becomes an $S\times R$-name).
  We need the following lemma.

  \begin{lemma}\label{sigmaccc}
    Let $T$ be a $\sigma$-closed forcing notion and $C$ be
    ccc. Then $$T\Vdash\check C \mbox{ is ccc}.$$
  \end{lemma}
  \begin{proof}
    Suppose not. Let $\{\dot c_\alpha:\alpha<\omega_1\}$ be
    a $T$-name for an antichain in $\check C$. Since $T$ is
    $\sigma$-closed, we can build a descending sequence
    $\langle t_\alpha\in T:\alpha<\omega_1\rangle$ and a
    sequence of conditions $\langle c_\alpha\in
    C:\alpha<\omega_1\rangle$ such that $$t_\alpha\Vdash\dot
    c_\alpha=\check c_\alpha.$$ But then $\{\check
    c_\alpha:\alpha<\omega_1\}$ is an uncountable antichain
    in $C$, a contradiction.
  \end{proof}

  Now, Lemma \ref{sigmaccc} implies that if $G$ is any
  $R$-generic over $V$, then in $V[G]$ we have $$S\Vdash
  {\dot Q\slash G}\mbox{ is ccc.}$$ This means that $R\Vdash
  ``S\Vdash \dot Q$ is ccc'', or in other words, $R\times
  S\Vdash\dot Q$ is ccc. Since $S\times R=R\times S$, by the
  remarks preceeding this theorem, we get that
  $\leq_{S\times R}\!\times\dot=$ witnesses uniform Axiom A
  for $(S\times R)*\dot Q$. 

  Therefore, there are $s'\leq_S s$ and $r'\leq_R r$ such
  that $(s',(r',\dot q))$ is compatible with only countably
  many elements of $A''$. By (b) above, $(s',r',\dot q)$ is
  compatible with only countably many elements of $A'$ and
  so is $(s',\pi(r',\dot q))$ since $A'\subseteq S\times P$.
  Since $(s',\pi'(p,(r',\dot q)))\leq_{S\times P}
  (s',\pi(r,\dot q))$ and by the definition of $\leq_{P'}$,
  we get that $(s',(p,(r',\dot q)))$ is compatible with only
  countably many elements of $A$. We also have
  $$(s',(p,(r',\dot q)))\leq_S\times\!\leq_0 (s,(p,(r,\dot
  q))),$$ hence $\leq_S\times \leq_0$ witnesses uniform
  Axiom A for $S\times P'$. This ends the proof of
  implication (ii)$\Rightarrow$(i).

  (i)$\Rightarrow$(ii). Suppose $P$ satisfies strong Axiom
  A. Since embeddability into \sigmaclosedccc\ is invariant
  under forcing-equivalence, we can assume that the ordering
  $\leq_0$ witnessing strong Axiom A is defined on $P$. We
  shall construct a $\sigma$-closed forcing notion $R$ and
  an $R$-name $\dot Q$ for a ccc forcing such that
  $P\lessdot R*\dot Q$. Let $R$ be the forcing with
  countable subsets of $P$ ordered as follows: for
  $\pi_0,\pi_1\subseteq P$ countable write $\pi_1\leq\pi_0$
  if
  \begin{itemize}
  \item for each $p\in\pi_0$ there is $q\in\pi_1$ such that
    $q\leq_0 p$,
  \item for each $q\in\pi_1$ the set $\pi_0$ is predense
    below $q$.
  \end{itemize}
  Note that $R$ is $\sigma$-closed.  In any $R$-generic
  extension the union of the countable subsets of $P$ which
  belong to the generic filter forms a suborder of $P$. Let
  $\dot Q$ be the canonical name for this subset. We will
  show that $P\lessdot R*\dot Q$ and that $\dot Q$ is forced
  to be ccc.

  \begin{lemma}
    The forcing $R*\dot Q$ adds a generic filter for $P$.
  \end{lemma}
  \begin{proof}
    We show that $R$ forces that the $\dot Q$-generic filter
    is $P$-generic over $V$. It is enough to show that for
    any dense open set $D\subseteq P$ and $p\in P$ the set
    $$\{\pi\in R: (\pi\Vdash p\notin\dot Q)\,\vee\,(\exists
    d\in\pi\ p\in D\,\wedge\, d\leq p)\}$$ is dense in $R$.
    Take any $\pi\in R$ and suppose $\pi\Vdash p\in\dot Q$.
    There is $\pi'\leq\pi$ and $p'\leq p$ such that
    $p'\in\pi'$. Pick $d\in D$ such that $d\leq p'$. Then
    $\pi'\cup\{d\}\leq\pi$ is as needed.
  \end{proof}

  Note now that for any $\pi\in R$ we have
  \begin{displaymath}\label{predense}
    \pi\Vdash\pi\mbox{ is predense in }\dot Q.\tag{$\star$}
  \end{displaymath}
  Indeed, if $\pi'\leq\pi$ and $\pi'\Vdash p\in\dot Q$, then
  there is $\pi''\leq\pi'$ and $q\in\pi''$ such that $q\leq
  p$. Since $\pi''\leq\pi$, there is $r\in\pi$ and $t\leq
  r,q$. Now $\pi''\cup\{t\}\Vdash t\leq r,q$.
  
  We will be done once we prove the following.

  \begin{lemma}
    $R$ forces that $\dot Q$ is ccc.
  \end{lemma}
  \begin{proof}
    Suppose that $\dot A$ is an $R$-name for an uncountable
    antichain in $\dot Q$. Assume that $\dot A$ is forced to
    be of cardinality $\omega_1$, namely $R\Vdash\dot
    A=\{\dot a_\alpha:\alpha<\omega_1\}$.
    \begin{sublemma}\label{sublemma}
      For each $\pi\in R$ and $p\in\pi$ there are
      $\pi'\leq\pi$, $p'\leq_0 p$ such that $p'\in\pi'$ and
      a countable $A_p\subseteq P$ such that $$\pi'\Vdash
      \{a\in\dot A: a\mbox{ \textnormal{is incompatible
          with} } p'\}\subseteq A_p.$$
    \end{sublemma}
    \begin{proof}
      We build an antichain in $R\times P$.  Let
      $C_0\subseteq R$ be a maximal antichain below $\pi$
      deciding $\dot a_0$ and such that for every $\rho\in
      C_0$ there is $b^{\rho}\in\rho$ such that
      $b^{\rho}\leq a$, where $a\in P$ is such that
      $\rho\Vdash a=\dot a_0$. Let $D_0 =\{(\rho, b^{\rho})
      : \rho\in C_0\}$. For $\xi<\omega_1$ use the fact that
      $R$ is $\sigma$-closed to find a maximal antichain
      $C_\xi$ below $\pi$ which refines all $C_\alpha$ for
      $\alpha<\xi$, decides $\dot a_\xi$ and for every
      $\rho\in C_\xi$ there is $b^{\rho}\in\rho$ such that
      $b^{\rho}\leq a$, where $a\in P$ is such that
      $\rho\Vdash a=\dot a_\xi$. Let $D_\xi =\{(\rho,
      b^{\rho}):\rho\in C_\xi\}$.

      Now $D=\bigcup_{\xi<\omega_1}D_\xi$ is an antichain in
      $R\times P$. To see that it is enough to check that if
      $\xi_0<\xi_1$, $(\rho_0, b^{\rho_0})\in D_{\xi_0}$,
      $(\rho_1, b^{\rho_1})\in D_{\xi_1}$ and $\rho_1\leq
      \rho_0$, then $b^{\rho_0}$ and $b^{\rho_1}$ are
      incompatible in $P$. Suppose $c\leq b^{\rho_0},
      b^{\rho_1}$ and put $\rho= \rho_1\cup\{c\}$. Then
      $$\rho\Vdash c\in\dot Q\mbox{ and } c\leq b^{\rho_1},
      b^{\rho_0}$$ and hence $\rho\Vdash\dot a_{\xi_0},\dot
      a_{\xi_1}$ are compatible. This is a contradiction.

      Since $R\times P$ satisfies uniform Axiom A via
      $\leq\times\leq_0$, we can find $\sigma\leq\pi$,
      $p'\leq_0 p$ and a countable subset $D'\subseteq D$
      such that $$\{(\rho,a)\in D: (\rho,a)\mbox{ is
        incompatible with }(\sigma,p')\}\subseteq D'.$$ Let
      $A_p=\{a\in P: \exists \rho\in R\ (\rho,a)\in D'\}$.
      Put $\pi'=\sigma\cup\{p'\}$.
    \end{proof}
    Take now any $\pi\in R$. Using Sublemma \ref{sublemma}
    and a bookkepping argument we find a sequence
    $\langle\pi_n\in R:n<\omega\rangle$ such that
    $\pi_0=\pi$ and for each $n<\omega$ and $p\in\pi_n$
    there is $m_p>n$, $p'\in\pi_{m_p}$ such that $p'\leq_0
    p$ and there is a countable $A_p\subseteq P$ such that
    \begin{displaymath}\label{antichain}
      \pi_m\Vdash \{a\in\dot A: a\mbox{ is incompatible with
      } p'\}\subseteq
      A_p.\tag{$\star\star$}
    \end{displaymath}

    For each $p\in\bigcup_{n<\omega}\pi_n$ construct a
    sequence $p_n\in P$ such that $p_0=p'\in\pi_{m_p}$ and
    if $p_n\in\pi_m$, then $p_{n+1}\in\pi_{m+1}$ is such
    that $p_{n+1}\leq_0 p_n$. Let $r_p$ be any condition
    such that $r_p\leq_0 p_n$ for all $n<\omega$..

    We define $\pi_\omega$ as the family of all such $r_p$
    for $p\in\bigcup_{n<\omega}\pi_n$. Note that
    $\pi_\omega\leq\pi_n$ for each $n$ and by
    (\ref{predense}) and (\ref{antichain}) we have that
    $$\pi_\omega\Vdash \bigcup\{A_p:
    p\in\bigcup_{n<\omega}\pi_n\}\mbox{ is predense in }\dot
    Q.$$ This contradicts the assumption that $\dot A$ is
    forced to be uncountable.
  \end{proof}   
  This ends the proof of the implication
  (ii)$\Rightarrow$(i).
\end{proof}

Now we prove Corollary \ref{original}.

\begin{proof}[Proof of Corollary \ref{original}]
  Recall the example \cite[Chapter XVII, Observation
  2.12]{shelah:proper} of two proper forcing notions whose
  product collapses $\omega_1$.  The first of them is
  $\sigma$-closed and the other is an iteration of the form
  ccc\,$*$\,$\sigma$-closed\,$*$\,ccc. Thus, the latter does
  not satisfy strong Axiom A but is forcing-equivalent to an
  Axiom A forcing, since it is $<\!\!\omega_1$-proper. It is
  not embeddable into \sigmaclosedccc\ by Theorem
  \ref{strongaa}.
\end{proof}

\section{Remaining questions}
\label{sec:questions}

There are a couple of questions which this papers leaves
open.

\begin{question}
  Is \textnormal{BFA(\sigmaclosedccc)} equivalent to
  \textnormal{\baafa}?
\end{question}

\begin{question}
  Is strong Axiom A equivalent to the fact that the product
  with every $\sigma$-closed forcing is
  $<\!\!\omega_1$-proper?
\end{question}

\begin{question}
  Does Theorem \ref{separationprinciples} hold for
  \textnormal{WCG$_\alpha$} in place of
  \textnormal{TWCG$_\alpha$}?
\end{question}

\begin{question}
  Does \textnormal{BFA(\sigmaclosedccc)} imply
  $2^{\aleph_0}=\aleph_2$?
\end{question}

\bibliographystyle{plain}
\bibliography{refs}

\end{document}